\documentclass[12pt, a4paper, oneside, reqno]{amsart}
\usepackage[numbered]{bookmark}
\usepackage{changepage} 
\usepackage[utf8]{inputenc}
\usepackage{graphics, graphicx}
\usepackage{amsmath, amsfonts, amsthm, amssymb}
\usepackage{mathrsfs, mathtools, mathabx}
\usepackage{verbatim} 
\usepackage{enumerate}
\usepackage{xcolor}
\usepackage{cite}

\newtheorem{theorem}{Theorem}[section]

\newtheorem{lemma}[theorem]{Lemma}
\newtheorem{proposition}[theorem]{Proposition}

\newtheorem{remark}{Remark}

\newcommand{\A}{\mathcal{A}} 
\newcommand{\cR}{\mathcal{R}}
\newcommand{\cZ}{\mathcal{Z}}

\newcommand{\Dom}{X_\A}

\newcommand{\AffDom}{\textup{Aff}_\A}
\newcommand{\AffDomC}{\textup{Aff}_{\A, \infty}}
\newcommand{\Exc}{\mathcal{E}_\A}

\newcommand{\Saff}{\textup{Aff}^1(\pi, \lambda, \omega)}

\newcommand{\N}{\mathbb{N}}

\newcommand{\R}{\mathbb{R}}

\newcommand{\Function}[5]{\begin{array}{cccc} #1 : & #2 & \rightarrow & #3 \\ & #4 & \mapsto & #5 \end{array}}
\date{}
\author[F.~Trujillo]{Frank Trujillo}
\address{Centre de Recerca Matemàtica, 08193 Bellaterra,
Barcelona, Spain}
\email{ftrujillo@crm.cat}
\begin{document}
\sloppy
\title{On the uniqueness of affine IETs semi-conjugated to IETs}
\maketitle
\begin{abstract}
We prove that for almost every irreducible interval exchange transformation $T$ and for any vector $\omega$ in its associated central-stable space (with respect to the Kontsevich-Zorich cocycle) there exists a unique AIET, up to normalization of its domain,  semi-conjugated to $T$ and whose log-slope vector equals $\omega$. This provides a partial answer to a question raised by S. Marmi, P. Moussa, and J.-C. Yoccoz in \cite{marmi_affine_2010}.
\end{abstract}

\section{Introduction}


\emph{Interval exchange transformations} (IETs) are bijective piecewise translations of a bounded interval with a finite number of discontinuities. They appear naturally as first-return maps of so-called \emph{locally Hamiltonian} flows on surfaces and have been the subject of extensive study in the last decades. First return maps associated with more general flows on surfaces lead to the so-called \emph{generalized IETs} (GIETs), which are bijective piecewise smooth maps with a finite number of discontinuities and positive derivative. As a convention, we assume the GIETs to be right continuous.

GIETs are sometimes seen as a natural generalization of circle diffeomorphisms, as this latter class appears naturally as first return maps of flows in the torus, where the set of IETs plays the role of circle rotations. Many recent works have explored to what extent certain aspects of the theory of circle diffeomorphisms extend to the GIET context (see, e.g., \cite{levitt_decomposition_1987, forni_solutions_1997, marmi_cohomological_2005, marmi_affine_2010, marmi_linearization_2012, ghazouani_local_2021, ghazouani_priori_2023}.) 

Let us mention some of the basic similarities between the theory of GIETs and circle diffeomorphisms, which play an important role in the present work. 

There exists a well-developed renormalization theory for IETs based on the so-called \emph{Rauzy-Veech induction} \cite{rauzy_echanges_1979, veech_gauss_1982}, which naturally extends to GIETs (see, e.g., \cite{marmi_affine_2010, marmi_linearization_2012}), and that generalizes the classical Gauss renormalization procedure for circle maps.  To each \emph{infinitely renormalizable} GIET (see \S \ref{sc:RV}) we can associate a semi-conjugacy invariant known as the \emph{combinatorial rotation number} or, simply, \emph{rotation number} (see \S \ref{sc:combinatorial_rot}), which generalizes the notion of rotation number in the circle diffeomorphisms setting.  Similar to the circle case, where every rotation number can be identified with a unique circle rotation, \emph{almost every} combinatorial rotation number (see \cite[\S 2.1]{berk_rigidity_2024} for a definition of almost every in this context) can be uniquely identified with an IET. Moreover, a GIET whose rotation number belongs to the full-measure set of \emph{$\infty$-complete} rotation numbers is semi-conjugated to a unique IET (see \S \ref{sc:combinatorial_rot}), and it is conjugated if, in addition, the GIET has no wandering intervals.

However, in contrast to the circle diffeomorphism setting, regularity alone is not enough to guarantee the absence of wandering intervals, see \cite{cobo_piece-wise_2002, marmi_affine_2010, cobo_wandering_2018}. Furthermore, smooth linearization of GIETs (i.e., the existence of a smooth conjugacy with an IET) does not depend only on the regularity and rotation number of the transformations. Moreover, important restrictions are needed for this conjugacy to exist (see, e.g., \cite{marmi_linearization_2012, ghazouani_local_2021, ghazouani_priori_2023}). 

The class of \emph{affine IETs} (AIETs), that is, piecewise linear GIETs, are an intermediate class between IETs and GIETs (which also appear as first return maps of certain flows on so-called \emph{dilation surfaces}, see \cite{duryev_dilation_2019}) whose behavior with respect to the Rauzy-Veech induction can be more easily understood than for GIETs (see, e.g., equations \eqref{eq:cocycle_relation_lengths} and \eqref{eq:acc_cocycle_relation_lengths} which the describe the change on the length and slope vectors after induction) but that keep some of the complexity of GIETs. Indeed, although in some settings smooth conjugacy with an IET is not possible, large classes of GIETs are smoothly conjugated to AIETs (see, e.g., \cite{cunha_rigidity_2014, berk_rigidity_2024}). 

In \cite{marmi_affine_2010}, S. Marmi, P. Moussa, and J.-C. Yoccoz showed that wandering intervals are a common feature among certain classes of AIETs. To study this phenomenon, the authors first analyzed the set of AIETs, with fixed slopes, semi-conjugated to a given IET. They provided necessary and sufficient conditions for this set to be non-empty and showed that if one normalizes the domain of the maps in this set to the unit interval, then this set can be parametrized as a compact simplex in a finite-dimensional space.

 More precisely, if $T = (\pi, \lambda) \in  \mathfrak{G}^0_\A \times \Delta_\A$ is an infinitely renormalizable \emph{irreducible} IET on $[0, 1)$ and $\omega \in \R^\A$, where $\A$ is a finite set (see \S \ref{sc:IETs} for the relevant notations), then the set $\Saff$ of AIETs on $[0, 1)$ semi-conjugated to $T$ and whose \emph{log-slope vector} (i.e., the vector given by the logarithm of the slopes on each of the exchanged intervals) is $\omega$ verifies the following.
 \begin{itemize}
 \item $\Saff \neq \emptyset$ if and only if $\langle \omega, \lambda \rangle = 0$.
 \item The set $\Saff$, parametrized using the \emph{length vector} (i.e.,  the vector given by the length of its exchanged intervals) of its elements, is a compact subsimplex of $\Delta_\A$.
 \end{itemize}
For a proof of these facts, see \cite[Proposition 2.3]{marmi_affine_2010}.  

Heuristic arguments (see \cite[Remarks 2.1, 3.2]{marmi_affine_2010}) led the authors in \cite{marmi_affine_2010} to suggest that, for almost every IET $(\pi, \lambda) \in  \mathfrak{G}^0_\A \times \Delta_\A$ and any $\omega \in \R^\A$ satisfying $\langle \omega, \lambda \rangle = 0$, the set $\Saff$ should consist of exactly one element. 

The goal of this article is to give a partial answer to the question above, by showing that the conclusion is correct in the case where the log-slope vector belongs to a certain non-empty subspace of positive codimension associated with the IET $(\pi, \lambda)$, namely, to the \emph{central-stable space} $E_{cs}(\pi, \lambda)$ associated with the \emph{Zorich cocycle} (see \S \ref{sc:lyapunov} for the relevant definitions). 

More precisely, our aim is to show the following.

\begin{theorem}
\label{thm: uniqueness}
For a.e. $(\pi, \lambda) \in \mathfrak{G}^0_\A \times \Delta_\A $ and for any $\omega \in E_{cs}(\pi, \lambda)$ there exists a unique AIET on $[0, 1)$ with log-slope $\omega$ semi-conjugated to $(\pi, \lambda)$, that is, $\#\Saff = 1$.
\end{theorem}

%

%
%
%


Theorem \ref{thm: uniqueness} will hold for any IET verifying the so-called \emph{Bounded Central Condition} (see Proposition \ref {prop:generic_condition}) which was introduced by  C. Ulcigrai and the author in \cite{trujillo_affine_2024}.

Let us point out that a similar result concerning an exceptional zero-measure class of IETs, known as IETs with \emph{bounded combinatorics}, had been previously obtained by K. Cunha and D. Smania in \cite[Proposition 3.10]{cunha_rigidity_2014}. Also, let us mention that the results in \cite{trujillo_affine_2024} show that, for a.e. IET $(\pi, \lambda)$ and for any $\omega$ as in Theorem \ref{thm: uniqueness}, any AIET in $\Saff$ is necessarily conjugated to $(\pi, \lambda)$. 

To prove our main result, we make use of a characterization of the set $\Saff$ introduced in \cite{marmi_affine_2010}, which is inspired by a classical argument to show unique ergodicity for almost every IET, and that reduces the question to understanding the intersection of certain simplicial cones associated with the AIET and showing that it reduces to a one-dimensional set. We describe this in detail at the beginning of Section \ref{sc:main}.

Although the authors in \cite{cunha_rigidity_2014} make use of the same observation, the arguments to conclude that the intersection reduces to a one-dimensional set rely heavily on the  bounded combinatorics assumption whereas our method relies on ergodic properties of the (accelerated) Rauzy-Veech renormalization and, in particular, in the so-called \emph{Bounded Central Condition} introduced by C. Ulcigrai and the author in \cite{trujillo_affine_2024}.



\section{Affine IETs and Rauzy-Veech induction} 
\label{sc:background}
In this section, we introduce the notations and basic properties for the spaces of IET/AIETs and some of the properties of the (accelerated) Rauzy-Veech induction/renormalization which will be used in this work. 

To avoid repetition, and since a big part of this background material appears already in many recent works concerning IETs, we will mostly limit ourselves to introducing the notations and describing Rauzy-Veech induction directly for AIETs.  Let us point out that, allthough initially defined only for IETs, the classical Rauzy-Veech induction extends naturally to AIETs (and, more generally, to GIETs, see, e.g., \cite{ghazouani_priori_2023}).  We will prove some properties concerning the Rauzy-Veech procedure for AIETs, but we refer to \cite{berk_rigidity_2024}, from where we borrow most of our notation, for more details on the classical Rauzy-Veech renormalization for IETs and other objects associated with it.

\subsection{The spaces of IETs and AIETs} 
\label{sc:IETs}
Throughout this paper, we fix $d \geq 2$ and a finite alphabet $\A$ with $d$ elements. We denote by $\Delta_\A$ the simplex of vectors in $\R_+^\A$ such that $|\lambda|:= |\lambda|_1 = 1$ and by $\mathfrak{G}^0_\A$ the set of \emph{irreducible permutations} on $\A$, which we encode as tuples  of the form $(\pi_0, \pi_1)$, where $\pi_0, \pi_1: \A \to \{1, \dots, d\}$ are bijective maps satisfying $\pi_1 \circ \pi_0^{-1}(\{1, \dots, k\}) \neq \{1, \dots, k\}$, for any $1 \leq k < d$.

\subsubsection*{Interval Exchange Transformations (IETs)} 
We can parametrize the space of \emph{irreducible} IETs (i.e., IETs whose permutation is irreducible) using the set $\mathfrak{G}^0_\A \times \R^\A_+$. More precisely, each $(\pi, \lambda) \in \mathfrak{G}^0_\A \times \R^\A_+$ can be uniquely identified with a partition of $[0, |\lambda|)$ into subintervals $\{I_\alpha\}_{\alpha \in \A}$, with lengths given by the \emph{length vector} $\lambda$ and ordered according to $\pi_0$, and a bijective right continuous piecewise translation $T = T_{(\pi, \lambda)}: [0, |\lambda|) \to [0, |\lambda|)$ which exchanges the subintervals in the partition according to the \emph{permutation} $\pi$.

Similarly, the set $\mathfrak{G}^0_\A \times \Delta_\A$ encodes the \emph{normalized} irreducible IETs, that is, the set of irreducible IETs defined on $[0, 1)$.


\subsubsection*{Affine Interval Exchange Transformations (AIETs)} We can parametrize the space of irreducible AIETs using the set of triples $(\pi, \ell, \omega) \in \mathfrak{G}^0_\A \times  \R^\A_+ \times  \R^\A$ satisfying $\langle \ell, e^\omega \rangle = 1$, where as an abuse of donation we denote by $e^\omega \in \R_+^\A$ the vector obtained by applying the exponential map to each coordinate of $\omega$. More precisely, each such triple can be uniquely identified with a partition of $[0, |\lambda|)$ into subintervals $\{I_\alpha\}_{\alpha \in \A}$, with lengths given by the \emph{length vector} $\ell$ and ordered according to $\pi_0$, and a bijective right continuous piecewise linear map $f = f_{(\pi, \ell, \omega)}: [0, |\lambda|) \to [0, |\lambda|)$ which modifies linearly the lengths of these subintervals using the \emph{slope vector} $e^\omega$, and then exchanges them according to the \emph{permutation} $\pi$. 

Similarly, the set of triples in $(\pi, \ell, \omega) \in \mathfrak{G}^0_\A \times  \Delta_\A \times  \R^\A$ satisfying $\langle \ell, e^\omega \rangle = 1$ encodes the irreducible IETs on $[0, 1)$.

\begin{remark}
The choice of using the so-called \emph{log-slope} vector $\omega$ instead of the slope vector $e^\omega$ is only a matter of convenience justified by the behaviour of the log-slope vector with respect to the Kontsevich-Zorich cocycle $($see \S \ref{sc:cocycles}, in particular, equation \eqref{eq:acc_cocycle_relation_lengths}$)$.
\end{remark}

Notice that we can naturally identify the space of IETs with a subspace of AIETs given by the triples of the form $(\pi, \lambda, \overline{1})$, where $\overline{1} \in \R^\A_+$ is the vector having $1$ on each coordinate.

 In the following, to avoid confusion, when referring to an IET we denote its length vector using the letter $\lambda$ and, when referring to an AIET we denote its length vector using the letter $\ell$.


\subsection{Rauzy-Veech induction}
\label{sc:RV}

We denote by $\Exc$  the \emph{full-measure set} of IETs, with respect to the product of the Lebesgue measure on $\R^\A_+$ and the counting measure $d\pi$ on  $\mathfrak{G}^0_\A$, satisfying \emph{Keane's condition} (see \cite{keane_interval_1975}). We denote by $\Exc^1$ the \emph{normalized} version of this space, that is, the set obtained by normalizing the domain of elements in $\Exc$ to $[0, 1)$. Similarly, we denote by $\AffDom$ the subset of AIETs in $\R^\A_+ \times \R^\A$ satisfying \emph{Keane's condition} and by $\AffDom^1$ its normalized version.


We will denote by $\cR: \AffDom \to \AffDom $ the \emph{Rauzy-Veech induction} and, for any AIET $f = (\pi, \ell, \omega) \in \AffDom$, we denote its iterates by $\cR$ as
\[(\pi^n, \ell^n, \omega^n):= \cR^n(\pi, \ell, \omega)\qquad \text{ for any }  n \geq 0.\]
The Rauzy-Veech induction preserves the space of IETs and so, for any IET $(\pi, \lambda) \in \Exc$, we denote
\[(\pi^n, \lambda^n):= \cR^n(\pi, \lambda) := \cR^n(\pi, \lambda, \overline{1}), \qquad \text{ for any }  n \geq 0.\]
We denote the normalized version of this operator by $\tilde \cR : \AffDom^1 \to \AffDom^1$, obtained by normalizing the domain of the AIET given by $\cR$ to $[0, 1)$.  It follows from the classical works of H. Masur \cite{masur_interval_1982} and W. Veech \cite{veech_gauss_1982} that $\tilde \cR: \Exc^1 \to \Exc^1$ admits a unique (modulo a multiplicative constant) invariant infinite measure $\mu_{\cR}$ equivalent to the product of Lebesgue and couting measures. 


\subsubsection*{The induction procedure} 
Given $f = (\pi, \ell, \omega) \in \AffDom$ with exchanged intervals $\{I_\alpha\}_{\alpha \in \A}$, we define its \emph{type} as
\[ \epsilon(f)  := \begin{cases} 0  \quad \text{ if } \quad |I_{\alpha_0}| > |f(I_{\alpha_1})|, \\ 1 \quad  \text{ if } \quad |I_{\alpha_0}| < |f(I_{\alpha_1})|, \end{cases}\]
where we denote $\alpha_{\epsilon} := \pi_\epsilon^{-1}(d)$, for any $\epsilon \in \{0, 1\}$. Notice that Keane's condition guarantees that $ |I_{\alpha_0}| \neq |f(I_{\alpha_1})|$. The symbols $\alpha_{\epsilon(f)}$ and $\alpha_{ 1 -\epsilon(f)}$ are referred to as the \emph{loser} and \emph{winner} symbols of $f$ respectively. We denote the type of the iterates of $f$ by $\cR$ as
\[\epsilon^n = \epsilon(\cR^n(f)),\qquad \text{ for any }  n \geq 0.\]

The \emph{Rauzy-Veech induction} $\cR(f)$ is given by inducing $f$ to the subinterval obtained by removing from $[0, |\ell|)$ either the last of the exchanged intervals (namely $I_{\alpha_0}$) or the image of the interval that becomes the last (namely $f(I_{\alpha_{1}})$), whichever is smaller. It is easy to check that the map $\cR(f)$ is an AIET with the same number of exchanged intervals of $f$ and that this operator preserves the space of IETs. Moreover, Keane's condition guarantees that the map $\cR$ can be iterated infinitely many times on $\AffDom$. 



\subsection{Combinatorial rotation number}
\label{sc:combinatorial_rot}
The projected action of the Rauzy-Veech induction defines a directed graph structure on $\mathfrak{G}^0_\A$, where each vertex has exactly two incoming and two outgoing edges which the type of the IETs can label. Any invariant subset $\mathfrak{R}\subseteq \mathfrak{G}^0_\A$ under this projected action is called a \emph{Rauzy graph}. 

Any (finite or infinite) admissible sequence of permutations in the directed graph structure of $\mathfrak{R}$ is called a \emph{Rauzy path}. The infinite path in the Rauzy graph defined by an AIET $f$ satisfying Keane's condition is called its \emph{combinatorial rotation number}, which we denote by $\gamma(f)$. 

\subsubsection*{Infinite completeness and semi-conjugacies} The combinatorial rotation number of an infinitely renormalizable IET is \emph{infinitely complete} (\emph{$\infty$-complete}), i.e., each symbol in the alphabet appears as the winner symbol infinitely many times.  Moreover, there exists at least one IET verifying Keane's condition associated with any $\infty$-complete path, and this correspondence is one-to-one when restricted to uniquely ergodic IETs (see Corollary 5 and Proposition 6 in \cite{yoccoz_echanges_2005}).   By \cite[Proposition 7]{yoccoz_echanges_2005}, any AIET satisfying Keane's condition and whose combinatorial rotation number is $\infty$-complete is semi-conjugated to an IET with the same combinatorial rotation number.

Given $(\pi, \lambda) \in \mathfrak{G}^0_\A \times \Delta_\A $ and $\omega \in \R^\A$, we denote 
\[ \Saff := \{ f \in \AffDom \mid f \text{ is semi-conjugated to } (\pi, \lambda) \text{ and has log-slope } \omega\}.\]
 By \cite[Proposition 7]{yoccoz_echanges_2005}, it follows that
 \[ \Saff = \{ f \in \AffDom \mid \gamma(f) = \gamma(T) \}. \]

\subsection{Zorich acceleration} We denote by $\AffDomC$ the set of AIETs in $\AffDom$ whose rotation number is $\infty$-complete. Recall that by the remarks above, the rotation number of any IET in $\Exc$ is $\infty$-complete. In particular, 
\[ \{f \in \AffDom \mid \exists T \in \Exc \text{ s.t. $f$ is semi-conjugated to $T$}\} \subseteq \AffDomC.\]
We will denote by $\cZ: \AffDomC \to \AffDomC$ the  \emph{Zorich acceleration} of $\cR$, which is given by
\begin{equation}
\label{eq:Zorich_acc}
\cZ(f) = \cR^{z(f)}(f), \qquad \text{ where } z(f) := \max \{ k \geq 1 \mid \epsilon^i(f)=\epsilon(f), \text{ for } 0 \leq i \leq k \}.
\end{equation}
Notice that if $\gamma(f)$ is $\infty$-complete, then $z(\cR^n(f)) < +\infty$, for any $n \geq 0$, and we can define an increasing sequence $z_n(f)$ such that $\cZ^n(f) = \cR^{z_n(f)}(f)$, for any $n \geq 0$. We denote the normalized version of this operator by $\tilde \cZ : \AffDom^1 \to \AffDom^1$.

For any AIET $f = (\pi, \ell, \omega) \in \AffDom$  we denote its iterates by $\cZ$ as
\[(\pi^{(n)}, \ell^{(n)}, \omega^{(n)}):= \cZ^n(\pi, \ell, \omega), \qquad \text{ for any $n \geq 0$},\] and for any IET $(\pi, \lambda) \in \Exc$, we denote
\[ (\pi^{(n)}, \lambda^{(n)}):= \cZ^n(\pi, \lambda) := \cZ^n(\pi, \lambda, \overline{1}), \qquad \text{ for any $n \geq 0$}.\]
It follows from the work of A. Zorich \cite{zorich_finite_1996} that the operator $\tilde \cZ : \Exc^1 \to \Exc^1$ preserves a unique invariant probability measure $\mu_{\cZ}$ equivalent to the product of Lebesgue and counting measures. 



\subsection{The (twisted) Rauzy-Veech cocycle}
\label{sc:cocycles}

We can describe the length vectors of the subsequent iterates by $\cR$ of an AIET through the \emph{twisted Rauzy-Veech cocycle} 
\[\Function{A}{\AffDom}{GL(\A, \R)}{f = (\pi, \ell, \omega)}{A(\pi, \epsilon(f), \omega)},\] 
where
\begin{equation}
\label{eq:RV_cocycle}
A(\pi, \epsilon, \omega) := \left\{ \begin{array}{ll} \textup{Id} + \rho_{\alpha_{1 - \epsilon}}E_{\alpha_{\epsilon}, \alpha_{1 -\epsilon}}, &  \text{if }  \epsilon = 0,\\
\textup{Id} + E_{\alpha_{\epsilon}, \alpha_{1 -\epsilon}} +  (\rho_{\epsilon} - 1)E_{\alpha_{1 - \epsilon}, \alpha_{1 -\epsilon}}, &   \text{if }  \epsilon = 1,
\end{array} \right. \qquad \rho := e^\omega \in \R^\A_+,
\end{equation}
and $E_{\alpha, \beta}$ is the matrix with zero entries except for the $(\alpha, \beta)$ entry which equals $1$.

Notice that we recover the classical Rauzy-Veech cocycle if we restrict ourselves to the set of IETs, that is, whenever $\rho = \overline{1} \in \R^\A$, or, equivalently,  whenever $\omega = \overline{0}$. 

Since $A$ only depends on the permutation, type and log-slope vector (in particular, only on the permutation and type when restricted to IETs), it will be useful for us to introduce some notation depending on the combinatorial rotation number (which encodes permutation and types of the AIET and all of its iterates) and the log-slope vector. 

For any infinite Rauzy path $\gamma = \{(\pi^i, \epsilon^i)\}_{i \geq 0}$ and for any $0 \leq m < n$, we define
\[
\begin{array}{ll}
A_{m, n}(\gamma) := A_n(\pi^m, \epsilon^m, \overline{1}) \dots A(\pi^{n - 1}, \epsilon^{n - 1}, \overline{1}), &  \qquad A_n(\gamma):= A_{0, n}(\gamma), \\
A_{m, n}(\gamma, \omega) := A_n(\pi^m, \epsilon^m, A_m(\gamma)\omega) \dots A(\pi^{n - 1}, \epsilon^{n - 1}, A_{n}(\gamma)\omega), & \qquad  A_n(\gamma, \omega):= A_{0, n}(\gamma, \omega).
\end{array}
\]
Notice that if $\gamma = \gamma(f)$ for some $f = (\pi, \ell, \omega) \in \AffDom$ then, for any $0 \leq m < n$, 
\[ A_{m, n}(\gamma, \omega) = A(\cR^m(f)) \dots A(\cR^{n - 1}(f)).\]
\subsubsection*{Cocycle relation for the length and log-slope vectors.} Given an AIET $f = (\pi, \ell, \omega) \in \AffDom$ and using the notations introduced in Section \ref{sc:RV} for the iterates of $f$ by $\cR$, we have  
\begin{equation}
\label{eq:cocycle_relation_lengths}
\omega^{m} = A_{m, n}(\gamma(f))^T \omega^n,  \qquad  \ell^{m} =  A_{m, n}(\gamma(f), \omega)^{-1} \ell^n,
\end{equation}
for any $0 \leq m < n$. 

As in the definition of $\cZ$ using $\cR$ (see \S \ref{sc:RV}), we can accelerate the cocycle $A$ to describe the iterates by $\cZ$ of an AIET. The \emph{Kontsevich-Zorich cocycle} is given by 
 \[\Function{B}{\AffDomC}{GL(\A, \R)}{f = (\pi, \ell, \omega)}{A_{z(f)}(\pi, \ell, \omega)},\] 
where $z(f)$ is given by \eqref{eq:Zorich_acc}. 

As before, we can define matrices associated wtih Rauzy paths as follows.  To any $\infty$-complete Rauzy path $\gamma = \{(\pi^i, \epsilon^i)\}_{i \geq 0}$ we can associate an increasing sequence $(z_n(\gamma))_{n \geq 1}$ such that 
\[ \epsilon^i = \epsilon^j, \quad \text{ for any } z_{n}(\gamma) \leq i, j < z_{n + 1}(\gamma), \qquad \epsilon^{z_n(\gamma)} \neq  \epsilon^{z_{n + 1}(\gamma)},\]
for any $n \geq 1$. We define, for any $\omega \in \R^\A$ and any $0 \leq m < n$,
\[
\begin{array}{ll} 
B_{m, n}(\gamma) := A_{z_m(\gamma), z_n(\gamma)}(\gamma), &   \qquad B_n(\gamma):= B_{0, n}(\gamma),\\
B_{m, n}(\gamma, \omega) := A_{z_m(\gamma), z_n(\gamma)}(\gamma, \omega) & \qquad  B_{0, n}(\gamma, \omega):= B_{0, n}(\gamma, \omega).
\end{array}
\]
Then, for any  $f = (\pi, \ell, \omega) \in \AffDomC$, we have
\begin{equation}
\label{eq:acc_cocycle_relation_lengths}
\omega^{(m)} =  B_{m, n}(\gamma(f))^T \omega^{(n)},   \qquad  \ell^{(m)}  =B_{m, n}(\gamma(f), \omega)^{-1} \ell^{(n)},
\end{equation}
for any $0 \leq m < n$. 

It follows from the work of A. Zorich \cite{zorich_finite_1996} that the Kontsevich-Zorich cocycle $B$ admits an Osedelet's filtration with respect to $\cZ$ (see \S \ref{sc:lyapunov}). Notice that equation \eqref{eq:acc_cocycle_relation_lengths} hints to the importance of the position of the log-slope vector $\omega$ inside the Oseledet's filtration associated with $B$.

\subsubsection*{Simplicial cones} Since the cocycle $A$ is given by non-negative matrices, for any infinite Rauzy path $\gamma$ and any $\omega \in \R^\A$, 
\begin{equation}
\label{eq:nested_seq}
A_{n + 1}(\gamma, \omega)(\R_+^\A) = A_{n}(\gamma, \omega) A_{n, n + 1}(\gamma, \omega) (\R_+^\A) \subseteq  A_{n}(\gamma, \omega)(\R_+^\A) \subseteq \R^\A_+,
\end{equation}
for any $n\geq 0$,  and thus 
\[ \mathcal{C}(\gamma, \omega):=  \bigcap_{n \geq 0} A_{n}(\gamma, \omega)(\R_+^\A) \subseteq \R_+^\A\]
defines a simplicial cone. Notice that if $\gamma$ is $\infty$-complete, then 
\[ \mathcal{C}(\gamma, \omega) =  \bigcap_{n \geq 0} B_{n}(\gamma, \omega)(\R_+^\A).\]
Moreover, by \eqref{eq:cocycle_relation_lengths}, we have 
\begin{equation}
\label{eq:length_vector_cones}
\ell \in  \mathcal{C}(\gamma(f), \omega), \qquad \text{ for any } f = (\pi, \ell, \omega) \in \AffDom.
\end{equation}

\subsection{Oseledet's filtration of the Kontsevich Zorich cocycle}
\label{sc:lyapunov}

Consider the skew product over $\tilde \cZ$, (i.e., the normalized version of $\cZ$) associated to the Kontsevich-Zorich cocycle and restricted to the space of IETs, namely 
\[ \Function{\cZ_B}{\Exc^1 \times \R^\A}{\Exc^1 \times \R^\A}{(\lambda, \pi, v)}{\left(\tilde \cZ(\pi, \lambda), {}^TB(\pi, \lambda, \overline{1})v\right)}.\]

Notice that, for any $n \in \N$, 
\[ \cZ_B^n(\lambda, \pi, v) = \big(\lambda^{(n)}, \pi^{(n)}, B_n(\gamma(\pi, \lambda))^Tv\big).\]

By Oseledet's theorem and a classical result of A. Zorich \cite{zorich_finite_1996}, for $\mu_{\cZ}$-almost-every $(\pi, \lambda) \in \Exc^1$  there exist subspaces
\[ \{0\} \subsetneq E_s(\pi, \lambda) \subsetneq E_{cs}(\pi, \lambda) \subsetneq \R^\A,\]
invariant by $\cZ_B$, corresponding to the subspaces of negative and non-positive Lyapunov exponents for the cocycle. More precisely, 
\[E_s(\pi, \lambda) = \left\{ v \in \R^\A \,\left|\, \lim_{n \to +\infty}\frac{1}{n}\log\big| B_{n}(\pi, \lambda)^Tv\big| < 0 \right\}\right.,\]
\[E_{cs}(\pi, \lambda) = \left\{ v \in \R^\A \,\left|\, \lim_{n \to +\infty} \frac{1}{n}\log\big| B_{n}(\pi, \lambda)^Tv\big| \leq 0 \right\}\right.,\]
\[\R^\A \,\setminus\, E_{cs}(\pi, \lambda) = \left\{ v \in \R^\A \,\left|\, \lim_{n \to +\infty} \frac{1}{n}\log\big| B_{n}(\pi, \lambda)^Tv\big| > 0 \right\}\right..\]
Moreover, it follows from the combination of several classical works \cite{zorich_finite_1996, forni_deviation_2002, avila_simplicity_2007} that the Kontsevich-Zorich cocycle has $2g$ non-zero Lyapunov exponents, where $1 \leq g \leq \frac{d}{2}$ is a natural number uniquely determined by the permutaiton of the IET, of the form
$$\theta_{1}> \theta_2 > \dots > \theta_g>0 > -\theta_g > \dots -\theta_2>-\theta_1.$$
In particular, $\dim(E_{cs}(\pi, \lambda)) = d - g \geq \tfrac{d}{2}$.

\subsection{Additional properties of the cocycles} Let us finish this section by proving some of the properties concerning the (twisted) Rauzy-Veech cocycle that will be used in the remaining of this article. 
\begin{lemma}
\label{lem:matrices}
Let $\gamma$ be an infinite Rauzy path, $\omega \in \R^\A$ wand $0 \leq m < n$. Then the following holds. 
\begin{enumerate}[(i)]
\item $A_{m, n}(\gamma, \omega)_{\alpha \alpha} > 0$, for any $\alpha \in \A$.
\item If $A_{m, n}(\gamma)_{\alpha\beta} > 0$ then $A_{m, n}(\gamma, \omega)_{\alpha\beta} > 0,$ for any $\alpha, \beta \in \A$.
\item \label{cond:bounded_coeff} For any $\alpha, \beta \in \A$, $A_{m, n}(\gamma, \omega)_{\alpha\beta}$ is a polynomial on $\rho:= \log \omega$ of degree at most $n - m$ with non-negative integer coefficients uniformly bounded by $A_{m, n}(\gamma)_{\alpha, \beta}$.
\end{enumerate}
\end{lemma}
\begin{proof}

Let us start by proving the second assertion using induction. Fix $m \geq 0$. Then, if $n = m + 1$ the assertion follows directly from the definition of the twisted Rauzy-Veech cocycle.  Let us assume that the second statement is true for some $n > m$ and let us show that it remains true for $n + 1$. Suppose that  $A_{m, n + 1}(\gamma)_{\alpha\beta} > 0$ for some $\alpha, \beta \in \A$. Then, as $A_{m, n + 1}(\gamma) = A_{m, n}(\gamma)A_{n, n + 1}(\gamma)$, 
\[ A_{m, n + 1}(\gamma)_{\alpha \beta} =  \sum_{\delta \in \A} A_{m, n}(\gamma)_{\alpha \delta} A_{n, n + 1}(\gamma)_{\delta \beta}  > 0.\]
Since the terms in the last equation are all non-negative, there exists $\sigma \in \A$ such that $ A_{m, n}(\gamma)_{\alpha \sigma}, A_{n, n + 1}(\gamma)_{\sigma \beta}  > 0$. Notice that $ A_{n, n + 1}(\gamma)_{\sigma \beta}  > 0$ implies that $A_{n, n + 1}(\gamma, \omega)_{\sigma\beta} > 0$ by definition of the twisted Rauzy-Veech cocycle. Moreover,   by the induction hypothesis, $A_{m, n}(\gamma)_{\alpha \sigma}  > 0$ implies that $ A_{m, n}(\gamma, \omega)_{\alpha \sigma} > 0.$ Therefore
\[ A_{m, n + 1}(\gamma, \omega)_{\alpha \beta} =  \sum_{\delta \in \A} A_{m, n}(\gamma, \omega)_{\alpha \delta} A_{n, n + 1}(\gamma, \omega)_{\delta \beta} \geq  A_{m, n}(\gamma, \omega)_{\alpha \sigma}  A_{n, n + 1}(\gamma, \omega)_{\sigma \beta} > 0.\]
This proves the second assertion.

To prove the first assertion, notice that $A_{m, n}(\gamma)_{\alpha \alpha} > 0$ for any $0 \leq m < n$ and any $\alpha \in \A$, since the terms in the diagonal of the classic Rauzy-Veech cocycle are all equal to $1$. Thus, the second assertion follows directly from the first one. 

We now prove the third assertion. Let us fix $\alpha, \beta \in \A$ and $0 \leq m < n$. Since $A_{m, n}(\gamma, \omega)$ is the product of $n - m$ matrices of the form \eqref{eq:RV_cocycle}, it follows directly that $A_{m, n}(\gamma, \omega)_{\alpha\beta}$ is a polynomial $p_{\alpha\beta}(\log \omega)$, on $\# \A$ variables, of degree at most $n - m$ with integer coefficients. Finally, noticing that the coefficients of $p_{\alpha\beta}$ are positive and that $p_{\alpha\beta}(\overline{1}) = A_{m, n}(\gamma)_{\alpha\beta}$, it follows that each coefficient is bounded by $A_{m, n}(\gamma)_{\alpha\beta}$.
\end{proof}


\section{Proof of the main result}
\label{sc:main}

Theorem \ref{thm: uniqueness} will hold for all Oseledet's generic IETs belonging to the full-measure set given below, which is a direct consequence of \cite[Proposition 3.8]{trujillo_affine_2024}.

\begin{proposition}[Bounded Central Condition]
\label{prop:generic_condition}
For a.e. $(\pi, \lambda) \in \mathfrak{G}^0_\A \times \Delta_\A $ there exists $V, K > 0$, $N \in \N$ and an increasing sequence $(n_k)_{k \geq 1} \subseteq \N$ such that, for any $k \geq 1$, the following holds.
\begin{enumerate}
\item \label{cond:bounded_central} $\left \|{}^TA_{n_k}(\pi, \lambda)\mid_{E_{cs}(\pi, \lambda)} \right \| \leq V.$
\item  \label{cond:positive_matrix}  The matrix $A_{n_k, n_k + N}(\pi, \lambda)$ is positive.
\end{enumerate}
\end{proposition}


Given \eqref{eq:length_vector_cones}, to prove Theorem \ref{thm: uniqueness} it suffices to show that, for any $(\pi, \lambda) \in \mathfrak{G}^0_\A \times \Delta_\A$ as in Proposition \ref{prop:generic_condition} and any $\omega \in E_{cs}(\pi, \lambda)$, the intersection
\begin{equation}
\label{eq:AIET_cone}
\mathcal{C}(\gamma, \omega) := \bigcap_{n \geq 0} A_{n}(\gamma, \omega)(\R_+^\A), \qquad \text{ where } \gamma:= \gamma(\pi, \lambda),
\end{equation}
reduces to a one-dimensional cone. 

Let us point out that a similar argument is used in a classical proof of the unique ergodicity of almost every IET. Indeed, it can be shown that the unique ergodicity of an IET $(\pi, \lambda)$ is equivalent to the nested intersection
\[  \mathcal{C}(\gamma) := \bigcap_{n \geq 0} A_{n}(\pi, \lambda)(\R_+^\A), \qquad \text{ where } \gamma:= \gamma(\pi, \lambda),\]
reducing to a one-dimensional cone (see, e.g., \cite[Chapter IV]{yoccoz_echanges_2005}). Moreover, let us mention that this implies the unique ergodicity of any $f \in \Saff$, for any $\omega \in \R^\A$, under a generic condition on $(\pi, \lambda)$ (see, e.g., \cite[Chapter III, \S 8]{yoccoz_echanges_2005}). However, the arguments to prove unique ergodicity through simplicial cones in the IET setting do not transfer to the AIET case since the latter case requires to consider the twisted Rauzy-Veech cocycle rather than the classical one.

We can define a \emph{projective metric} on $\R^\A_+$ as
\[ \textup{d}_{\textup{p}}(u, v):= \log \sup\left. \left\{ \frac{u_\alpha}{v_\beta} \frac{v_\alpha}{u_\beta}\,\right|\,  \alpha, \beta \in \A \right\},\]
for any $u, v \in \R^\A_+$. Notice that, although as an abuse of notation we refer to $\textup{d}_{\textup{p}}$ as a metric, it does not define a metric on $\R^\A_+$. Indeed, although $\textup{d}_{\textup{p}}$ is non-negative, symmetric and verifies the triangular inequality (see, e.g., \cite[Lemma 26.1]{viana_ergodic_2006}), for any  $u, v \in \R^\A_+$, we have
\[  \textup{d}_{\textup{p}}(u, v) = 0 \quad \text{ if and only if } \quad \exists t > 0, \, u = tv.\]

Nevertheless, notice that $\textup{d}_{\textup{p}}$ does define a metric on $\Delta_\A$ and that
\begin{equation}
\label{eq:projected_distance}
\textup{d}_{\textup{p}}(P_\A(v), P_\A(w)) = \textup{d}_{\textup{p}}(u, v), \qquad \text{ for any } u, v \in \R^\A_+,
\end{equation}
where $P_\A: \R^\A_+ \to \Delta_\A$ denotes the projection $x \mapsto \tfrac{x}{|x|_1}$.

In the following, given a non-negative matrix $M \in GL(\A, \R)$ we denote by $\kappa(M)$ its contraction factor in the projective metric of $\R^\A_+$, that is,
\[ \kappa(M):= \inf \left. \left\{ \frac{\textup{d}_{\textup{p}}(Mv, Mw)}{\textup{d}_{\textup{p}}(v,w)} \,\right|\, \text{ for any } v, w \in \R^\A_+ \text{ with } \textup{d}_{\textup{p}}(v,w) \neq 0 \right\}.\]
The following are classic facts about non-negative matrices and projective metrics. For a proof, we refer the reader to \cite[\S 26]{viana_ergodic_2006}.
\begin{lemma}
\label{lem:contraction_matrices}
Let $\A$ be a finite alphabet and let $\textup{d}_{\textup{p}}$ denote the projective metric on $\R^\A_+$. Then the following holds.
\begin{enumerate}[(i)]
\item $\kappa(M) \leq 1$, for any $M \in GL(\A, \R)$ non-negative.
\item $\kappa(M) < 1$, for any $M \in GL(\A, \R)$ positive. Moreover, for any $C > 0$ there exists $0 < \kappa < 1$ such that if $C^{-1} < M_{\alpha\beta} < C$, for all $\alpha, \beta \in \A$, then $\kappa(M) < \kappa$.
\item $\kappa(M_1M_2) \leq \kappa(M_1)\kappa(M_2),$  for any $M_1, M_2 \in GL(\A, \R)$ non-negative.
\end{enumerate}
\end{lemma}

We can now prove Theorem \ref{thm: uniqueness}.

\begin{proof}[Proof of Theorem \ref{thm: uniqueness}] 
Let $(\pi, \lambda) \in  \mathfrak{G}^0_\A \times \Delta_\A $ as in Proposition \ref{prop:generic_condition} and $\omega \in E_{cs}(\pi, \lambda)$. In view of  \eqref{eq:length_vector_cones}, to prove the Theorem it suffices to show that the intersection given by \eqref{eq:AIET_cone} reduces to a one-dimensional cone, or, equivalently, to show that
\begin{equation}
\label{eq:projected_cone}
P_\A \big( \mathcal{C}(\gamma, \omega) \big) =  P_\A \left( \bigcap_{n \geq 0} A_{n}(\gamma, \omega)(\R_+^\A) \right),
\end{equation}
reduces to a point, where $P_\A$ denotes the projection from $\R^\A_+$ onto $\Delta_\A$ and $\gamma:= \gamma(\pi, \lambda)$.

Let $N$ and $(n_k)_{k \geq 1}$ be given by Proposition \ref{prop:generic_condition}. Up to considering a subsequence we may assume, without loss of generality, that $n_k - n_{k - 1} \geq N$ for all $k \geq 1$. Let us denote $m_k:= n_k + N$, for any $k \geq 1$. 

By \eqref{cond:positive_matrix} in Proposition \ref{prop:generic_condition}, the matrix $A_{m_1}(\gamma) = A_{m_1}(\pi, \lambda)$ is positive. Thus, by Lemma \ref{lem:matrices}, the matrix $A_{m_1}(\gamma, \omega)$ is positive. Therefore 
\[\overline{A_{m_1}(\gamma, \omega)(\R^\A_+)} \subseteq \R^\A_+.\]
In particular, 
\[K:= P_\A \left( \overline{A_{m_1}(\gamma, \omega)(\R^\A_+)} \right) \subseteq \Delta_\A\]
 is a compact subset of $\Delta_\A$ and has finite diameter with respect to $\textup{d}_{\textup{p}}$, that is,
\[\sup_{v, w \in K} \textup{d}_{\textup{p}}(v, w) < D,\]
for some $D > 0$. 

Since the terms in the intersection \eqref{eq:projected_cone} form a nested sequence (see \eqref{eq:nested_seq}), to prove the theorem it suffices to show that
\begin{equation}
\label{eq:limit_contraction}
\lim_{k \to \infty} \kappa(A_{m_1, m_k}(\gamma, \omega)) = 0.
\end{equation}
Indeed, for any $v, w \in \bigcap_{n \geq 0} A_{n}(\gamma, \omega)(\R_+^\A)$ and any $k \geq 1$ there exist $v_k, w_k \in  A_{m_1}(\gamma, \omega)(\R_+^\A)$ such that $v = A_{m_1, m_k}(\gamma, \omega) v_k$ and $w = A_{m_1, m_k}(\gamma, \omega) w_k.$ Hence
\begin{align*}
\textup{d}_{\textup{p}}(v, w)  &= \textup{d}_{\textup{p}}(A_{m_1, m_k}(\gamma, \omega) v_k, A_{m_1, m_k}(\gamma, \omega) w_k) \\
& \leq \kappa(A_{m_1, m_k}(\gamma, \omega)) \textup{d}_{\textup{p}}(v_k, w_k) \\ &= \kappa(A_{m_1, m_k}(\gamma, \omega))  \textup{d}_{\textup{p}}(P_\A(v_k), P_\A(w_k))  \\
& \leq \kappa(A_{m_1, m_k}(\gamma, \omega))  D,
 \end{align*}
 where the third line follows from \eqref{eq:projected_distance}. Thus, assuming \eqref{eq:limit_contraction}, it follows that $P_\A(v) = P_\A(w)$,  for any $v, w  \in \bigcap_{n \geq 0} A_{n}(\gamma, \omega)(\R_+^\A)$, which implies that the projected simplicial cone given by \eqref{eq:projected_cone} reduces to a point.
 
Hence, to finish the proof of the theorem, let us show \eqref{eq:limit_contraction}. 

Notice that, for any $k \geq 1$, we can express the matrix $A_{m_1, m_{k + 1}}(\gamma, \omega)$ as 
\begin{align*}
A_{m_1, m_{k + 1}}(\gamma, \omega) & = A_{m_1, m_k}(\gamma, \omega)A_{m_k, m_{k + 1} - N}(\gamma, \omega)A_{m_{k + 1} - N, m_{k + 1}}(\gamma, \omega) \\
& =  A_{m_1, m_k}(\gamma, \omega)A_{m_k, n_{k + 1}}(\gamma, \omega)A_{n_{k + 1}, n_{k + 1} + N}(\gamma, \omega)
\end{align*} 
We claim that
\begin{equation}
\label{eq:uniform_contraction}
\kappa:= \sup_{k \geq 1} \kappa(A_{n_{k + 1}, n_{k + 1} + N}(\gamma, \omega))  < 1.
\end{equation}
Before showing \eqref{eq:uniform_contraction}, let us show how it implies \eqref{eq:limit_contraction}. Assuming \eqref{eq:uniform_contraction}, by Lemma \ref{lem:contraction_matrices}, it follows that 
\begin{align*}
\kappa(A_{m_1, m_{k + 1}}(\gamma, \omega)) & \leq \kappa(A_{m_1, m_k}(\gamma, \omega)) \kappa(A_{m_k, n_{k + 1}}(\gamma, \omega)) \kappa(A_{n_{k + 1}, n_{k + 1} + N}(\gamma, \omega)) \\ & \leq \kappa(A_{m_1, m_k})  \kappa,
\end{align*}
for any $k \geq 1$. Therefore, 
\[ \kappa(A_{m_1, m_{k + 1}}) \leq \kappa ^k \xrightarrow[k \to \infty]{} 0,\]
which finishes the proof of \eqref{eq:limit_contraction}. 

Finally, let us prove \eqref{eq:uniform_contraction}. Fix $k \geq 1$. By \eqref{cond:bounded_central}  in Proposition \ref{prop:generic_condition}, 
\[\big|\omega^{n_{k + 1}}\big| = \big| {}^TA_{n_k}(\gamma) \omega\big| \leq V |\omega|.\]
where $V$ is the constant in Proposition \ref{prop:generic_condition}.
Since, for any $0 \leq i \leq N$,  the matrices $A_{n_{k + 1}, n_{k + 1 + i}}(\gamma)$ are products of at most $N$ matrices whose entries 
 are uniformly bounded, it follows that
 \begin{equation}
 \label{eq:bounded_slopes}
 \big|\omega^{n_{k + 1} + i}\big| = \big|{}^TA_{n_{k + 1}, n_{k + 1 + i}}(\gamma) \omega^{n_{k + 1}}\big| \leq C_N |\omega^{n_{k + 1}}| \leq C_N V|\omega|,
 \end{equation}
 for any $0 \leq i \leq N$, where $C_N$ is a constant depending only on $N$.
 
By Proposition \ref{prop:generic_condition}, $A_{n_{k + 1}, n_{k + 1} + N}(\gamma)$ is a positive matrix. 
Thus, by Lemma \ref{lem:matrices},  the matrix $A_{n_{k + 1}, n_{k + 1} + N}(\gamma, \omega)$ is positive. Moreover, notice that  $A_{n_{k + 1}, n_{k + 1} + N}(\gamma, \omega)$ is the product of $N$ matrices whose entries can be uniformly bounded, from above and below, as functions of the norms of $\omega^{n_{k + 1}}, \dots, \omega^{n_{k + 1} + N}$ (see Condition \eqref{cond:bounded_coeff} in Lemma \ref{lem:matrices}). Hence, there exists a constant $\Gamma = \Gamma(V, C_N, \omega) > 0$ such that
\[ \Gamma^{-1} < A_{m_{k + 1} - N, m_{k + 1}}(\gamma, \omega)_{\alpha\beta} < \Gamma, \qquad \text{ for any } \alpha, \beta \in \A.\]
Therefore,  \eqref{eq:uniform_contraction} follows from the equation above and Lemma \ref{lem:contraction_matrices}. 
\end{proof}

\section*{Acknowledgments} 
This project has received funding from the European Union's Horizon 2020 research and innovation programme under the Marie Skłodowska-Curie grant agreement No. 101154283. 
\bibliographystyle{acm}
\bibliography{Uniqueness_of_AIET_Arxiv.bbl}
\end{document}